\documentclass[12pt,a4paper]{amsart}
\usepackage{hyperref,amsmath}
\usepackage{graphics}
\newtheorem{theorem}{Theorem}[section]

\newtheorem{corollary}[theorem]{Corollary}
\theoremstyle{theorem}

\newtheorem{lemma}[theorem]{Lemma}
\theoremstyle{definition}

\numberwithin{equation}{section} \makeatletter
\@namedef{subjclassname@2010}{ 2010 Mathematics Subject
Classification} \makeatother
\begin{document}
\title[On the construction of l-equienergetic graphs]{ON THE CONSTRUCTION OF L-EQUIENERGETIC GRAPHS}
\author{S. Pirzada and Hilal A Ganie}

\address{Department of Mathematics \\
    University of Kashmir \\
   Srinagar, Hazratbal 190006\
   \\ India}
 \email{pirzadasd@kashmiruniversity.ac.in; sdpirzada@yahoo.co.in}
 \email{hilahmad1119kt@gmail.com}
\date{}

\subjclass[2010]{0340-6253}
\begin{abstract}
\indent\indent For a graph with $n$ vertices and $m$ edges, having Laplacian spectrum $\mu_1, \mu_2, \cdots,\mu_n$ and signless Laplacian spectrum $\mu^+_1,\mu^+_2, \cdots,\mu^+_n$, the Laplacian energy and signless Laplacian energy of $G$ are respectively, defined as $LE(G)=\sum_{i=1}^{n}|\mu_i-\frac{2m}{n}|$ and $LE^+(G)=\sum_{i=1}^{n}|\mu^+_i-\frac{2m}{n}|$. Two graphs $G_1$ and $G_2$ of same order are said to be $L$-equienergetic if $LE(G_1)=LE(G_2)$ and $Q$-equienergetic if $LE^{+}(G_1)=LE^{+}(G_2)$. The problem of constructing graphs having same Laplacian energy has been considered by Stevanovic for threshold graphs and by Liu and Liu for those graphs whose order is $n\equiv 0$ (mod 7). In general the problem of constructing $L$-equienergetic graphs from any pair of given graphs is still not solved, and  this work is an attempt in that direction. We construct sequences of non-cospectral (Laplacian, signless Laplacian) $L$-equienergetic and $Q$-equienergetic graphs from any pair of graphs having same number of vertices and edges.
\end{abstract}
\keywords{Laplacian energy, spectra, signless Laplacian energy $L$-equienergetic, $Q$-equienergetic}
\subjclass[2010]{0340-6253}
\maketitle{}

\section{Introduction}\label{sec1}
\indent Let $G$ be finite, undirected, simple graph with $n$ vertices and $m$ edges having vertex set $V(G)=\{v_1, v_2, \cdots, v_n\}$. Throughout this paper we denote such a graph by $G(n,m)$. The adjacency matrix $A=(a_{ij})$ of $G$ is a $(0, 1)$-square matrix of order $n$ whose $(i,j)$-entry is equal to one if $v_i$ is adjacent to $v_j$ and equal to zero, otherwise. The adjacency spectrum of $G$ is the spectrum of its adjacency matrix. The energy of $G$ is the sum of the absolute values of the adjacency eigenvalues of $G$ (see \cite{g}). This quantity introduced by Gutman has noteworthy chemical applications (see \cite{gp}).\\
\indent Let $D(G)={diag}(d_1, d_2, \cdots, d_n)$ be the diagonal matrix associated to $G$, where $d_i$ is the degree of vertex $v_i$. The matrices $L(G)$=$D(G)$-$A(G)$ and $L^+(G)$=$D(G)$+$A(G)$ are called Laplacian and signless Laplacian matrices and the spectrum of the matrices $L(G)$ and $L^+(G)$ are called Laplacian spectrum ($L$-spectrum) and signless Laplacian spectrum ($Q$-spectrum) of $G$, respectively. Being real symmetric, positive semi-definite matrices, let $0=\mu_n\leq\mu_{n-1}\leq\cdots\leq\mu_1$ and $0\leq\mu^+_n\leq\mu^+_{n-1}\leq\cdots\leq\mu^+_1$ be the $L$-spectrum and $Q$-spectrum of $G$, respectively. It is well known $\mu_n$=0 with multiplicity equal to number of connected components of $G$ (see \cite{f}). Fiedler \cite{f} showed that a graph $G$ is connected if and only if its second smallest Laplacian eigenvalue is positive and called this as the algebraic connectivity of the graph $G$. Also it is well known that for a bipartite graph the $L$-spectra and $Q$-spectra coincides (see \cite{cs}). The Laplacian energy of a graph $G$ as put forward by Gutman and Zhou (see \cite{gz}) is defined as $LE(G)=\sum\limits_{i=1}^{n}|\mu_i-\frac{2m}{n}|$. This quantity, which is an extension of graph-energy concept has found remarkable chemical applications beyound the molecular orbital theory of conjugated molecules (see \cite{rg}). In analogy to Laplacian energy the signless Laplacian energy of $G$ is defined as $LE^+(G)=\sum\limits_{i=1}^{n}|\mu^+_i-\frac{2m}{n}|$. Both these graph-energy extensions have been extensively studied and can be seen in the literature (see \cite{a, acg, wh, zg, z} and the references therein). It is easy to see that $tr(L(G))=\sum\limits_{i=1}^{n}\mu_i=\sum\limits_{i=1}^{n-1}\mu_i=2m$ and $tr(LE^+(G))=\sum\limits_{i=1}^{n}\mu^+_i=2m$.\\
\indent Two graphs $G_1$ and $G_2$ of same order are said to be equienergetic if $E(G_1)$=$E(G_2)$, (see \cite{b} and \cite{rgra}). In analogy to this two graphs $G_1$ and $G_2$ of same order are said to $L$-equienergetic if $LE(G_1)$=$LE(G_2)$ and $Q$-equienergetic if $LE^+(G_1)$=$LE^+(G_2)$. In \cite{ll} and \cite{st}, the families of $L$-equienergetic graphs have been constructed for a particular class of graphs. As per the existing literature the problem of constructing $L$-equienergetic graphs from a general pair of given graphs is still open and present work is aimed in this direction. In this paper, we will show how sequences of $L$-equienergetic ($Q$-equienergetic) non-cospectral (Lapalcian, signless Laplacian) graphs can be constructed from any pair of graphs having same number of vertices and edges.\\
\indent We denote the complement of graph $G$ by $\bar{G}$, the complete graph on $n$ vertices by $K_n$, the empty graph by $\bar{K_n}$ and the complete bipartite graph with cardinalities of partite sets $q$, $r$ by $K_{q,r}$. The rest of the paper is organised as, in section 1 some preliminary results which are important through out the paper are presented, in section 2 construction of families of $L$-equienergetic ($Q$-equienergetic) by using the graph operation like, union, join and complement is given and in section 3 product of graphs is used for the constrution of $L$-equienergetic ($Q$-equienergetic) graphs.

\section{Preliminaries}
In this section, we define some graph operations with terminology from (see \cite{p}) together with their $L$-spectra and $Q$-spectra (see \cite{cds}) which are used throughout the paper.\\
\indent For the graphs $G_1$ and $G_2$ with disjoint vertex sets $V(G_1)$ and $V(G_2)$ the cartesian product is a graph $G=G_1\times G_2$ with vertex set $V(G_1)\times V(G_2)$ and an edge $((u_1, v_1), (u_2, v_2))$ if and only if $u_1=u_2$ and $(v_1, v_2)$ is an edge of $G_2$ or $v_1=v_2$
and $(u_1, u_2)$ is an edge of $G_1$. Following Lemma gives the $L$-spectra ($Q$-spectra) of the cartesian product of graphs.

\begin{lemma}\cite{cds}
Let $G_1(n_1,m_1)$ and $G_2(n_2,m_2)$ be two graphs having L-spectra (Q-spectra) $\mu_1, \mu_2, \cdots, \mu_{n_1}$ and $\sigma_1, \sigma_2, \cdots, \sigma_{n_2}$, respectively, then the L-spectra  (Q-spectra) of $G=G_1\times G_2$ is $\mu_i+\sigma_j$ where $i=1,2,\cdots,n_1$ and $j=1,2,\cdots,n_2$.
\end{lemma}

\indent The conjunction (Kronecker product) of $G_1$ and $G_2$ is a graph $G=G_1\otimes G_2$ with vertex set $V(G_1)\times V(G_2)$ and an edge $((u_1,v_1), (u_2, v_2))$ if and only if $(u_1, u_2)$ and $(v_1, v_2)$ are edges in $G_1$ and $G_2$, respectively. Following Lemma gives the $L$-spectra ($Q$-spectra) of the Kronecker product of graphs.
\begin{lemma}\cite{cds}
Let $G_1(n_1,m_1)$ and $G_2(n_2,m_2)$ be two graphs having L-spectra (Q-spectra) $\mu_1, \mu_2, \cdots, \mu_{n_1}$ and $\sigma_1, \sigma_2, \cdots, \sigma_{n_2}$, respectively, then the L-spectra  (Q-spectra) of $G=G_1\otimes G_2$ is $\mu_i\sigma_j$ where $i=1,2,\cdots,n_1$ and $j=1,2,\cdots,n_2$.
\end{lemma}
The join (complete product) of $G_1$ and $G_2$ is a graph $G=G_1\vee G_2$ with vertex set $V(G_1)\cup V(G_2)$ and an edge set consisting of all the edges of $G_1$ and $G_2$ together with the edges joining each vertex of $G_1$ with every vertex of $G_2$. The $L$-spectra of join of graphs is given by the following result.
\begin{lemma}\cite{cds}
Let $G_1(n_1,m_1)$ and $G_2(n_2,m_2)$ be two graphs having $L$-spectra respectively as $\mu_1,\mu_2,\cdots,\mu_{n_1-1},\mu_{n_1}=0$ and $\sigma_1, \sigma_2, \cdots, \sigma_{n_2-1}, \sigma_{n_2}=0$, then the $L$-spectra of $G=G_1\vee G_2$ is $n_1+n_2, n_1+\sigma_1, n_1+\sigma_2, \cdots, n_1+\sigma_{n_2-1}, n_2+\mu_1, n_2+\mu_2, \cdots, n_2+\mu_{n_1-1}, 0$.
\end{lemma}
\indent The disjoint union of graphs $G_1$ and $G_2$ is a graph $G$ with vertex set and edge set, respectively, equal to union of vertex sets and edge sets of $G_1$ and $G_2$. The complement of a graph $G$ is the graph $\bar{G}$ with vertex set same as $G$ and with two vertices adjacent in $\bar{G}$ if and only if they are non-adjacent in $G$. Lemmas (2.4) and (2.5) (see\cite{cds}) gives the spectra of disjoint union of graphs and $L$-spectra of complement of a graph.
\begin{lemma}
Let $G_1$ and $G_2$ be two vertex disjoint graphs, then the spectra (adjacency, Laplacian, signless Laplacian) of graph $G=G_1\cup G_2$ is the union of spectra of $G_1$ and $G_2$.
\end{lemma}
\begin{lemma}
Let $G(n,m)$ be a graph having L-spectra $\mu_1, \mu_2, \cdots, \mu_n=0$, then the L-spectra of complement $\bar{G}$ of $G$ is $n-\mu_1, n-\mu_2, \cdots, n-\mu_{n-1}, \mu_n=0$.
\end{lemma}
\indent If $G$ is any proper subgraph of $K_n$ then the $L$-spectra of $G$ (see \cite{fa}) is given by the following result.
\begin{lemma}
\indent Let $G(s,m)$ be a subgraph of $K_n$, then the $L$-spectra of the graph $K_n-E(G)$ is $n-\mu_1, n-\mu_2, \cdots, n-\mu_s, n ((n-s-1)-times), 0$.
\end{lemma}
\section{Construction of $L$-equienergetic ($Q$-equienergetic) graphs by means of operations, union, join and complement}
\indent In this section we will construct sequences of $L$-equienergetic and $Q$-equienergetic non-cospectral (Laplacian, signless Laplacian) graphs from a given pair of connected graphs having same number of vertices and edges by using the well known graph operations like, disjoint union, join and complement of graphs.\\
\indent Theorems 3.1 and 3.2 give the construction of a sequence of $L$-equienergetic ($Q$-equienergetic) graph pairs which are disconnected from a given pair of connected graphs.
\begin{theorem}
Let $G_1$ and $G_2$ be two connected graphs having $L$-spectra $0=\mu_n<\mu_{n-1}\leq\cdots\leq\mu_1$ and $0=\lambda_n<\lambda_{n-1}\leq\cdots\leq\lambda_1$, respectively. For a positive integer $p$ such that  $$\frac{2m}{p+n}<\min(\mu_{n-1}, \lambda_{n-1}),$$ we have $LE(G_1\cup\bar{K_p})=LE(G_2\cup\bar{K_p})$.
\end{theorem}
\begin{proof}
By Lemma 2.4, the $L$-spectra of the graphs $G_1\cup \bar{K_p}$ and $G_2\cup \bar{K_p}$ are respectively $\mu_1, \mu_2, \cdots, \mu_{n-1}, 0 ((p+1)-times)$ and $\lambda_1, \lambda_2, \cdots, \lambda_{n-1}, 0 ((p+1)-times)$, with average vertex degree $\dfrac{2m^{\prime}}{n^{\prime}}=\dfrac{2m}{p+n}.$
Therefore,
\begin{align*}
LE(G_1\cup\bar{K_p})&=\sum\limits_{i=1}^{n-1}|\mu_i-\frac{2m^{\prime}}{n^{\prime}}|+(p+1)|0-\frac{2m^{\prime}}{n^{\prime}}|\\&
=\sum\limits_{i=1}^{n-1}(\mu_i-\frac{2m}{n+p})+(p+1)\frac{2m}{p+n}\\&=2m+(p-n-2)\frac{2m}{p+n}.
\end{align*}
Also,
\begin{align*}
LE(G_2\cup\bar{K_p})&=\sum\limits_{i=1}^{n-1}|\lambda_i-\frac{2m^{\prime}}{n^{\prime}}|+(p+1)|0-\frac{2m^{\prime}}{n^{\prime}}|\\&
=\sum\limits_{i=1}^{n-1}(\lambda_i-\frac{2m}{n+p})+(p+1)\frac{2m}{p+n}\\&=2m+(p-n-2)\frac{2m}{p+n}.
\end{align*}
Clearly $LE(G_1\cup\bar{K_p})=LE(G_2\cup\bar{K_p})$.
\end{proof}
\indent A result similar to Theorem 3.1 can be put forward for the graphs $G_1$ and $G_2$ for the $Q$-spectra in the following way.
\begin{theorem}
Let $G_1$ and $G_2$ be two connected graphs having $Q$-spectra $0<\mu^+_n<\mu^+_{n-1}\leq\cdots\leq\mu^+_1$ and $0<\lambda^+_n<\lambda^+_{n-1}\leq\cdots\leq\lambda^+_1$, respectively. For a positive integer $p$ such that  $$\frac{2m}{p+n}<\min(\mu^+_{n-1}, \lambda^+_{n-1}),$$ we have  $LE^+(G_1\cup\bar{K_p})=LE^+(G_2\cup\bar{K_p})$.
\end{theorem}
\begin{proof}
The proof follows by the same argument as in Theorem 3.1.
\end{proof}
\indent For any two connected graphs with same number of vertices and edges, the next result gives the construction of connected graphs having the same $L$-energy.
\begin{theorem}
Let $G_1$ and $G_2$ be two connected graphs having $L$-spectra respectively $0=\mu_n<\mu_{n-1}\leq\cdots\leq\mu_1$ and $0=\lambda_n<\lambda_{n-1}\leq\cdots\leq\lambda_1$, with algebraic connectivity greater than one. For a positive integer $p$ such that  $$\frac{2m}{p+n}<\min(\mu_{n-1}-1, \lambda_{n-1}-1),$$ we have $LE(\bar{G_1}\vee K_p)$=$LE(\bar{G_2}\vee K_p)$.
\end{theorem}
\begin{proof}
The $L$-spectra of $K_p$ is $0, p ((p-1)-times)$, therefore by Lemmas (2.3, 2.5) it follows that the $L$-spectra of graphs $\bar{G_1}\vee K_p$ and $\bar{G_2}\vee K_p$ is $p+n-\mu_1, p+n-\mu_2, \cdots, p+n-\mu_{n-1}, (p+n) (p-times), 0$ and $p+n-\lambda_1, p+n-\lambda_2, \cdots, p+n-\lambda_{n-1}, (p+n) (p-times), 0$, respectively, with average vertex degree $\dfrac{2m^{\prime}}{n^{\prime}}=p+n-1-\dfrac{2m}{p+n}.$ So, for $i=1, 2, \cdots, n-1$, we have
$$p+n-\mu_i-\dfrac{2m^{\prime}}{n^{\prime}}=-(\mu_i-1-\dfrac{2m}{p+n})\leq0.$$
Similarly, $p+n-\lambda_i-\dfrac{2m^{\prime}}{n^{\prime}}\leq0.$ Therefore,
\begin{align*}
LE(\bar{G_1}\vee K_p)&=\sum\limits_{i=1}^{n-1}|p+n-\mu_i-\frac{2m^{\prime}}{n^{\prime}}|+p|p+n-\frac{2m^{\prime}}{n^{\prime}}|+|0-\dfrac{2m^{\prime}}{n^{\prime}}|\\&
=(n-p)(\dfrac{2m^{\prime}}{n^{\prime}})+(p+n)(p-n+1).
\end{align*}
Also,
\begin{align*}
LE(\bar{G_2}\vee K_p)&=\sum\limits_{i=1}^{n-1}|p+n-\lambda_i-\frac{2m^{\prime}}{n^{\prime}}|+p|p+n-\frac{2m^{\prime}}{n^{\prime}}|+|0-\dfrac{2m^{\prime}}{n^{\prime}}|\\&
=(n-p)(\dfrac{2m^{\prime}}{n^{\prime}})+(p+n)(p-n+1).
\end{align*}
Clearly, $LE(\bar{G_1}\vee K_p)$=$LE(\bar{G_2}\vee K_p)$.
\end{proof}
\indent Since the graph $G\vee K_p$ is the complement of the graph $G\cup\bar {K_p}$, we combine the Theorem 3.3 and 3.4 to obtain the following result.
\begin{theorem}
Let $G_1$ and $G_2$ be two connected graphs having $L$-spectra respectively$0=\mu_n<\mu_{n-1}\leq\cdots\leq\mu_1$ and $0=\lambda_n<\lambda_{n-1}\leq\cdots\leq\lambda_1$, with algebraic connectivity greater than one. For a positive integer $p$ such that  $$\frac{2m}{p+n}<\min(\mu_{n-1}-1, \lambda_{n-1}-1),$$ we have $LE(\bar{G_1}\vee K_p)$=$LE(\bar{G_2}\vee K_p)$ and $LE(G_1\cup\bar{K_p})=LE(G_2\cup\bar{K_p})$.
\end{theorem}
\begin{corollary}
Let $G_1(s,m)$ and $G_2(s,m)$ be two connected proper subgraphs of the complete graph $K_n$ having $L$-spectra $0=\mu_s<\mu_{s-1}\leq\cdots\leq\mu_1$ and $0=\lambda_s<\lambda_{s-1}\leq\cdots\leq\lambda_1$, respectively, with algebraic connectivities greater than one. For a positive integer $p$ such that  $$\frac{2m}{p+n}<\min(\mu_{s-1}-1, \lambda_{s-1}-1),$$ we have $LE(\tilde{G_1})$=$LE(\tilde{G_2})$, where $\tilde{G_i}=(K_n-E(G_i))\vee K_p$ for $i=1, 2$.
\end{corollary}
\begin{proof}
The result can be proved by using the same argument as in Theorem 3.4 and the fact that the $L$-spectra of $K_n-E(G_1)$ and $K_n-E(G_2)$ are $n-\mu_1, n-\mu_2, \cdots, n-\mu_s, n ((n-s-1)-times), 0$ and $n-\lambda_1, n-\lambda_2, \cdots, n-\lambda_s, n ((n-s-1)-times), 0$, respectively, with average vertex degree $\dfrac{2m^{\prime}}{n^{\prime}}=p+n-1-\dfrac{2m}{p+n}$.
\end{proof}
\indent If $k$ is the first value of positive integer $p$ satisfying the condition given in the Theorems 3.1 to 3.5, then every integer greater than $k$ also satisfies this condition, hence we will obtain sequences of graphs having same $L$-energy ($Q$-energy). In Theorems 3.6 and 3.8, we present another way of construction of graphs having the same $L$-energy from a given pair of connected graphs.
\begin{theorem}
Let $G_1$ and $G_2$ be two connected graphs having $L$-spectra respectively $0=\mu_n<\mu_{n-1}\leq\cdots\leq\mu_1$ and $0=\lambda_n<\lambda_{n-1}\leq\cdots\leq\lambda_1$,.
For a positive integer $p\geq n$ such that
$$\frac{2m}{p+n}<\min(\mu_{n-1}, \lambda_{n-1}),$$ we have $LE(G_1\vee \bar{K_p})$=$LE(G_2\vee \bar{K_p}).$
\end{theorem}
\begin{proof}
The $L$-spectra of the graphs $G_1\vee \bar{K_p}$ and $G_2\vee \bar{K_p}$ (by Lemma 2.3) are
$p+\mu_1, p+\mu_2, \cdots, p+\mu_{n-1}, (p+n) ((p-1)-times), 0$ and
$p+\lambda_1, p+\lambda_2, \cdots, p+\lambda_{n-1}, (p+n) ((p-1)-times), 0$
respectively, with average vertex degree $\dfrac{2m^{\prime}}{n^{\prime}}=\dfrac{2m}{p+n}+\dfrac{2pn}{p+n}.$ So, for $i=1, 2, \cdots, n-1$, we have\\
$$p+\mu_i-\dfrac{2m^{\prime}}{n^{\prime}}=\dfrac{p(p-n)}{p+n}+\mu_i-\dfrac{2m}{p+n}\geq 0.$$
Similarly, $p+\lambda_i-\dfrac{2m^{\prime}}{n^{\prime}}\geq 0.$ Therefore,
\begin{align*}
LE(G_1\vee \bar{K_p})&=\sum\limits_{i=1}^{n-1}|p+\mu_i-\frac{2m^{\prime}}{n^{\prime}}|+(p-1)|n-\frac{2m^{\prime}}{n^{\prime}}|\\&+|p+n-\dfrac{2m^{\prime}}{n^{\prime}}|+|0-\dfrac{2m^{\prime}}{n^{\prime}}|\\&
=(p-n)(\dfrac{2m^{\prime}}{n^{\prime}})+2(m+n).
\end{align*}
Also,
\begin{align*}
LE(G_2\vee \bar{K_p})&=\sum\limits_{i=1}^{n-1}|p+\lambda_i-\frac{2m^{\prime}}{n^{\prime}}|+(p-1)|n-\frac{2m^{\prime}}{n^{\prime}}|\\&+|p+n-\dfrac{2m^{\prime}}{n^{\prime}}|+|0-\dfrac{2m^{\prime}}{n^{\prime}}|\\&
=(p-n)(\dfrac{2m^{\prime}}{n^{\prime}})+2(m+n).
\end{align*}
Clearly, $LE(G_1\vee \bar{K_p})$=$LE(G_2\vee \bar{K_p})$.
\end{proof}
\indent If in Theorem 3.6, the complete graph $K_p$ is replaced by the graph $K_{p,p}$, we will obtain another family of $L$-equienergetic graphs as seen in the following observation.
\begin{corollary}
Let $G_1$ and $G_2$ be two connected graphs having $L$-spectra respectively $0=\mu_n<\mu_{n-1}\leq\cdots\leq\mu_1$ and $0=\lambda_n<\lambda_{n-1}\leq\cdots\leq\lambda_1$.
For a positive integer $p\geq n$ such that
$$\frac{2m}{p+n}<\min(\mu_{n-1}, \lambda_{n-1}),$$ we have $LE(G_1\vee \bar{K_{p,p}})$=$LE(G_2\vee \bar{K_{p,p}})$.
\end{corollary}
\begin{proof}
The proof follows by the same argument as in Theorem 3.6.
\end{proof}
\begin{theorem}
Let $G_1(n,m)$ and $G_2(n,m)$ be two connected graphs having $L$-spectra respectively $0=\mu_n<\mu_{n-1}\leq\cdots\mu_1$ and $0=\lambda_n<\lambda_{n-1}\leq\cdots\lambda_1$.
For a positive integer $p\geq n+4$ such that
$$\dfrac{2m}{p+n}<\min(\mu_{n-1}, \lambda_{n-1}),$$
we have $LE(\bar{G_1}\cup K_p)=LE(\bar{G_2}\cup K_p)$.
\end{theorem}
\begin{proof}
The $L$-spectra of the graphs $\bar{G_1}\cup K_p$ and $\bar{G_2}\cup K_p$ (by Lemma 2.4, 2.5) are
$n-\mu_1, n-\mu_2, \cdots, n-\mu_{n-1}, p((p-1)-times), 0, 0$ and
$n-\lambda_1, n-\lambda_2, \cdots, n-\lambda_{n-1}, p((p-1)-times), 0, 0,$ respectively,
with average vertex degree $\dfrac{2m^{\prime}}{n^{\prime}}=\dfrac{n^2+p^2}{p+n}-\dfrac{2m}{p+n}-1.$ So, for $i=1, 2, \cdots, n-1$, we have
$$n-\mu_i-\dfrac{2m^{\prime}}{n^{\prime}}=\dfrac{p(n-p)}{p+n}-(\mu_i-1-\dfrac{2m}{p+n})\leq0.$$
Similarly, $n-\lambda_i-\dfrac{2m^{\prime}}{n^{\prime}}=\dfrac{p(n-p)}{p+n}-(\lambda_i-1-\dfrac{2m}{p+n})\leq0.$ Therefore,
\begin{align*}
LE(\bar{G_1}\cup K_p)&=\sum\limits_{i=1}^{n-1}|n+\mu_i-\frac{2m^{\prime}}{n^{\prime}}|+(p-1)|p-\frac{2m^{\prime}}{n^{\prime}}|+2|0-\frac{2m^{\prime}}{n^{\prime}}|\\&
=(n-1)(\dfrac{2m^{\prime}}{n^{\prime}}-n)+(p-1)|p-\frac{2m^{\prime}}{n^{\prime}}|+\dfrac{4m^{\prime}}{n^{\prime}}-2m.
\end{align*}
Also,
\begin{align*}
LE(\bar{G_2}\cup K_p)&=\sum\limits_{i=1}^{n-1}|n+\lambda_i-\frac{2m^{\prime}}{n^{\prime}}|+(p-1)|p-\frac{2m^{\prime}}{n^{\prime}}|+2|0-\frac{2m^{\prime}}{n^{\prime}}|\\&
=(n-1)(\dfrac{2m^{\prime}}{n^{\prime}}-n)+(p-1)|p-\frac{2m^{\prime}}{n^{\prime}}|+\dfrac{4m^{\prime}}{n^{\prime}}-2m \\&=LE(\bar{G_1}\cup K_p).
\end{align*}
\end{proof}
\indent The following observation is an easy consequence of Theorem 3.8.
\begin{corollary}
Let $G_1(s,m)$ and $G_2(s,m)$ be two connected proper subgraphs of the complete graph $K_n$ having $L$-spectra respectively $0=\mu_s <\mu_{s-1}\leq\cdots\leq\mu_1$ and $0=\lambda_s<\lambda_{s-1}\leq\cdots\leq\lambda_1$, with algebraic connectivity greater than one. For a positive integer $p\geq n$ such that  $$\frac{2m}{p+n}<\min(\mu_{s-1}-1, \lambda_{s-1}-1),$$ we have $LE(\tilde{G_1})=LE(\tilde{G_2})$, where $G_i=(K_n-E(G_i))\cup K_p$ for $i=1,2$.
\end{corollary}
\indent Since the graph $\bar{G}\cup K_p$ is the complement of the graph $G\vee \bar{K_p}$, we can combine Theorems 3.6 and 3.8 to obtain the following result.
\begin{theorem}
Let $G_1(n,m)$ and $G_2(n,m)$ be two connected graphs having $L$-spectra respectively $0=\mu_n<\mu_{n-1}\leq\cdots\mu_1$ and $0=\lambda_n<\lambda_{n-1}\leq\cdots\lambda_1$.
For a positive integer $p\geq n+4$ such that
$$\dfrac{2m}{p+n}<\min(\mu_{n-1}, \lambda_{n-1}),$$
we have $LE(\bar{G_1}\cup K_p)=LE(\bar{G_2}\cup K_p)$ and $LE(G_1\vee \bar{K_p})$=$LE(G_2\vee \bar{K_p})$.
\end{theorem}
\indent Again, if $k$ is the first value of positive integer $p$ satisfying the condition given in the Theorems 3.6 to 3.10, then every integer greater than $k$ also satisfies this condition, hence we will obtain a sequence of $L$-equienergetic graph pairs. Note that such a sequence is possible as we are dealing with finite graphs.\\
\section{Construction of $L$-equienergetic ($Q$-equienergetic) graphs by using product of graphs}
\indent In section 3, we have seen how the operations like union, join and complement of graphs can be used to obtain sequences of $L$-equienergetic graph from any given pair of connected graphs. In this section we will use product of graphs togather with the operations introduced in Section 3 to obtain new families of $L$-equienergetic ($Q$-equienergetic) graphs from any given pair of connected graphs having same number of vertices and edges.
\begin{theorem}
Let $G_1(n,m)$ and $G_2(n,m)$ be two connected graphs having $L$-spectra respectively $0=\mu_n<\mu_{n-1}\leq\cdots\leq\mu_1$ and $0=\lambda_n<\lambda_{n-1}\leq\cdots\leq\lambda_1$.
For a positive integer $p> n$ such that
$$\frac{2m}{p+n}<\min(\mu_{n-1}, \lambda_{n-1}),$$ we have $LE(\tilde{G_1})$=$LE(\tilde{G_2})$, where $\tilde{G_i}=(G_i\cup \bar{K_p})\times K_p$, $i=1, 2$.
\end{theorem}
\begin{proof}
The $L$-spectra of the graphs $G_1\cup \bar{K_p}$ and $G_2\cup \bar{K_p}$ are respectively
$\mu_1, \mu_2, \cdots, \mu_{n-1}, \mu_n=\mu_{n+1}=\cdots=\mu_{p+n}=0$ and
$\lambda_1, \lambda_2, \cdots, \lambda_{n-1}, \lambda_n=\lambda_{n+1}=\cdots=\lambda_{p+n}=0,$.
Also the $L$-spectra of $K_p$ is $p=\tau_1=\tau_2=\cdots=\tau_{p-1}, \tau_p=0.$
So by Lemma 1.1, the $L$-spectra of graphs $\tilde{G_1}$ and $\tilde{G_2}$ are respectively
$\mu_i+\tau_j$ and $\lambda_i+\tau_j$, where $1\leq i \leq {p+n}, 1\leq j \leq p,$
with average vertex degree $\dfrac{2m^{\prime}}{n^{\prime}}=\dfrac{2m}{n+p}+p-1.$
So, for $i=1, 2, \cdots, n-1$, we have
$$p+\mu_i-\dfrac{2m^{\prime}}{n^{\prime}}=\mu_i-\dfrac{2m}{n+p}+1> 0$$ and
$$\mu_i-\dfrac{2m^{\prime}}{n^{\prime}}=\mu_i-\dfrac{2m}{p+n}-p+1\leq 0.$$
Similarly, $p+\lambda_i-\dfrac{2m^{\prime}}{n^{\prime}}> 0$ and
$\lambda_i-\dfrac{2m^{\prime}}{n^{\prime}}\leq 0.$ Therefore,
\begin{align*}
&LE(\tilde{G_1})=\sum\limits_{i=1}^{p+n}\sum\limits_{j=1}^{p}|\mu_i+\tau_j-\dfrac{2m^{\prime}}{n^{\prime}}|\\&
=(p-1)\sum\limits_{i=1}^{p+n}|p+\mu_i-\dfrac{2m^{\prime}}{n^{\prime}}|+\sum\limits_{i=1}^{p+n}|\mu_i-\dfrac{2m^{\prime}}{n^{\prime}}|\\&
=(p-1)\sum\limits_{i=1}^{n-1}(p+\mu_i-\dfrac{2m^{\prime}}{n^{\prime}})+(p-1)\sum\limits_{i=n}^{p+n}|p-\dfrac{2m^{\prime}}{n^{\prime}}|+\sum\limits_{i=1}^{p+n}(\dfrac{2m^{\prime}}{n^{\prime}}-\mu_i)\\&
=(p-1)(n-1)(p-\dfrac{2m^{\prime}}{n^{\prime}})+(p-1)(p+n)|p-\dfrac{2m^{\prime}}{n^{\prime}}|\\&+(p+n)\dfrac{2m^{\prime}}{n^{\prime}}+2m(p-2).
\end{align*}
Also
\begin{align*}
&LE(\tilde{G_2})=\sum\limits_{i=1}^{p+n}\sum\limits_{j=1}^{p}|\lambda_i+\tau_j-\dfrac{2m^{\prime}}{n^{\prime}}|\\&
=(p-1)(n-1)(p-\dfrac{2m^{\prime}}{n^{\prime}})+(p-1)(p+n)|p-\dfrac{2m^{\prime}}{n^{\prime}}|\\&+(p+n)\dfrac{2m^{\prime}}{n^{\prime}}+2m(p-2).
\end{align*}
Hence the result follows.
\end{proof}
\indent A result similar to Theorem 4.1 can be put forward for the connected graphs $G_1$ and $G_2$ for the $Q$-spectra in the following way.
\begin{theorem}
Let $G_1(n,m)$ and $G_2(n,m)$ be two connected graphs having $Q$-spectra respectively $0<\mu^+_n\leq\mu^+_{n-1}\leq\cdots\leq\mu^+_1$ and $0<\lambda^+_n\leq\lambda^+_{n-1}\leq\cdots\leq\lambda^+_1$, with $\mu^+_n, \lambda^+_n> 1$.
For a positive integer $p$ such that  $$\frac{2m}{p+n}<\min(\mu^+_{n}-1, \lambda^+_{n}-1),$$ we have $LE^+((G_1\cup\bar{K_p})\times K_p)=LE^+((G_2\cup\bar{K_p})\times K_p)$.
\end{theorem}
\begin{proof}
The $Q$-spectra of the graphs $(G_1\cup\bar{K_p})\times K_p$ and $(G_2\cup\bar{K_p})\times K_p$ are (by Lemmas 1.1, 1.4) $2p-2+\mu^+_1, 2p-2+\mu^+_2, \cdots, 2p-2+\mu^+_n, (2p-2) (p-times), \{p-2+\mu^+_1, \cdots, p-2+\mu^+_n\} (each (p-1)-times), (p-2) (p(p-1)-times)$ and
$2p-2+\lambda^+_1, 2p-2+\lambda^+_2, \cdots, 2p-2+\lambda^+_n, (2p-2) (p-times), \{p-2+\lambda^+_1, \cdots, p-2+\lambda^+_n\} (each (p-1)-times), (p-2) (p(p-1)-times),$ respectively,
with average vertex degree $\dfrac{2m^{\prime}}{n^{\prime}}=\dfrac{2m}{n+p}+p-1.$
So for $i=1, 2, \cdots, n-1$, we have
$$2p-2+\mu^+_i-\dfrac{2m^{\prime}}{n^{\prime}}=p+\mu^+_i-\dfrac{2m}{n+p}-1> 0$$ and
$$p-2+\mu^+_i-\dfrac{2m^{\prime}}{n^{\prime}}=\mu^+_i-\dfrac{2m}{p+n}-1> 0.$$
Similarly, $2p-2+\lambda^+_i-\dfrac{2m^{\prime}}{n^{\prime}}> 0$ and
$p-2+\mu^+_i-\dfrac{2m^{\prime}}{n^{\prime}}> 0.$\\
The results now follows by proceeding in the same way as in Theorem 4.1.
\end{proof}
\indent Theorem 4.2 is true for connected graphs $G_1$ and $G_2$ having algebraic connectivity greater than one, however if we consider graph $(G\cup \bar{K_p})\times K_{p,p}$ in place of the graph $(G\cup \bar{K_p})\times K_p$, it holds for any connected pair of graphs as seen in the following.
\begin{corollary}
Let $G_1$ and $G_2$ be two connected graphs having $Q$-spectra respectively $0<\mu^+_n\leq\mu^+_{n-1}\leq\cdots\leq\mu^+_1$ and $0<\lambda^+_n\leq\lambda^+_{n-1}\leq\cdots\leq\lambda^+_1$.
For a positive integer $p\geq 2n$ such that  $$\frac{2m}{p+n}<\min(\mu^+_{n}, \lambda^+_{n}),$$ we have $LE^+((G_1\cup\bar{K_p})\times K_{p,p})=LE^+((G_2\cup\bar{K_{p,p}})\times K_p)$.
\end{corollary}
\begin{proof}
The proof follows by the same argument as in Theorem (4.2) and the fact that the $Q$-spectra of the graph $K_{p,p}$ is $p, \dfrac{p}{2} ((p-2)-times), 0$.
\end{proof}
\indent The following corollary is a consequence of the Theorem 4.1.
\begin{corollary}
Let $G_1(s,m)$ and $G_2(s,m)$ be two connected proper subgraphs of the complete graph $K_n$ having $L$-spectra respectively $0=\mu_s<\mu_{s-1}\leq\cdots\leq\mu_1$ and $0=\lambda_s<\lambda_{s-1}\leq\cdots\leq\lambda_1$, with algebraic connectivity greater than two. For a positive integer $p$ such that  $$\frac{2m}{p+n}<\min(\mu_{n-1}-2, \lambda_{n-1}-2),$$ we have $LE(\tilde{G_1})$=$LE(\tilde{G_2})$, where $G_i=((K_n-E(G_i))\cup K_p)\times K_p$ for $i=1, 2$.
\end{corollary}
\begin{theorem}
Let $G_1$ and $G_2$ be two connected graphs having $L$-spectra respectively $0=\mu_n<\mu_{n-1}\leq\cdots\leq\mu_1$ and $0=\lambda_n<\lambda_{n-1}\leq\cdots\leq\lambda_1$.
For a positive integer $p\geq 2n$ such that
$$\frac{2m}{p+n}<\min(\mu_{n-1}, \lambda_{n-1}),$$ we have $LE(\tilde{G_1})$=$LE(\tilde{G_2})$, where $\tilde{G_i}=(\bar{G_i}\cup K_p)\times K_p$, $i=1, 2$.
\end{theorem}
\begin{proof}
The $L$-spectra of the graphs $\bar{G_1}\cup K_p$ and $\bar{G_2}\cup K_p$ are respectively as
$\gamma_1, \gamma_2, \cdots, \gamma_{p+n}$ and $\theta_1, \theta_2, \cdots, \theta_{p+n},$ where
$$\gamma_i=\left\{\begin{array}{lr}n-\mu_1, &\mbox{if $1\leq i \leq {n-1}$};\\
0, &\mbox{if $i=n, p+n$};\\
p, &\mbox{if $n+1\leq i \leq{p+n-1}$}.
\end{array} \right.$$
and
$$\theta_i=\left\{\begin{array}{lr}n-\lambda_i, &\mbox{if $1\leq i \leq{n-1}$};\\
0, &\mbox{if $i=n, p+n$};\\
p, &\mbox{if $n+1\leq i \leq{p+n-1}$}
\end{array} \right.$$
Also $L$-spectra of $K_p$ is
$p=\tau_1=\tau_2=\cdots=\tau_{p-1}, 0=\tau_p.$
So by Lemma 1.1, the $L$-spectra of $\tilde{G_1}$ and $\tilde{G_2}$ are
$\gamma_i+\tau_j$ and $\theta_i+\tau_j$, where $1\leq i \leq{p+n}, 1\leq j \leq{p}$
with average vertex degree $\dfrac{2m^{\prime}}{n^{\prime}}=p-2-\dfrac{2m}{n+p}+\dfrac{p^2+n^2}{p+n}.$
For $i=1, 2, \cdots, n-1$, we have
$$n-\mu-i-\dfrac{2m^{\prime}}{n^{\prime}}=\dfrac{2(p+n)-2p^2}{p+n}+\dfrac{2m}{p+n}-\mu_i\leq 0$$ and
$$p+n-\mu_i-\dfrac{2m^{\prime}}{n^{\prime}}=\dfrac{p(n-p)+2(p+n)}{p+n}+\dfrac{2m}{p+n}-\mu_i\leq 0.$$
Similarly, $n-\lambda_i-\dfrac{2m^{\prime}}{n^{\prime}}\leq 0$ and
$p+n-\lambda_i-\dfrac{2m^{\prime}}{n^{\prime}}\leq 0.$\\
\indent\indent It is now easy to see that $LE(\tilde{G_1})$=$LE(\tilde{G_2})$.
\end{proof}
\indent In this section, till now, we have obtained families of disconnected graphs having have same $L$-energy ($Q$-energy) from any given pair of connected graphs by using cartesian product of graphs. In the next part of this section we will use cartesian product to construct families of connected graphs having same $L$-energy.
\begin{theorem}
Let $G_1$ and $G_2$ be two connected graphs having $L$-spectra respectively $0=\mu_n<\mu_{n-1}\leq\cdots\leq\mu_1$ and $0=\lambda_n<\lambda_{n-1}\leq\cdots\leq\lambda_1$, with algebraic connectivity greater than two.
For a positive integer $p$ such that
$$\frac{2m}{p+n}<\min(\mu_{n-1}-2, \lambda_{n-1}-2),$$ we have $LE(\tilde{G_1})$=$LE(\tilde{G_2})$, where $\tilde{G_i}=(\bar{G_i}\vee{K_p})\times K_p$, $i=1, 2$.
\end{theorem}
\begin{proof}
The $L$-spectra of the graphs $\bar{G_1}\vee K_p$ and $\bar{G_2}\vee K_p$ are respectively as
$\gamma_1, \gamma_2, \cdots, \gamma_{p+n} $ and $\theta_1, \theta_2, \cdots, \theta_{p+n},$ where
$$\gamma_i=\left\{\begin{array}{lr}p+n-\mu_1, &\mbox{if $1\leq i \leq {n-1}$};\\
p+n, &\mbox{if $n\leq i\leq{p+n-1}$};\\
0, &\mbox{if $i=p+n$}
\end{array} \right.$$
and
$$\theta_i=\left\{\begin{array}{lr}p+n-\lambda_i, &\mbox{if $1\leq i \leq {n-1}$};\\
p+n, &\mbox{if $n\leq i\leq{p+n-1}$};\\
0, &\mbox{if $i=p+n$}.
\end{array} \right.$$
Also $L$-spectra of $K_p$ is
$p=\tau_1=\tau_2=\cdots=\tau_{p-1}, 0=\tau_p.$
So by Lemma 1.1, the $L$-spectra of $\tilde{G_1}$ and $\tilde{G_2}$ are
$\gamma_i+\tau_j$ and $\theta_i+\tau_j$, where $1\leq i \leq{p+n}, 1\leq j \leq p$
with average vertex degree $\dfrac{2m^{\prime}}{n^{\prime}}=2p+n-2-\dfrac{2m}{n+p}.$
For $i=1, 2, \cdots, n-1$, we have\\
$$2p+n-\mu_i-\dfrac{2m^{\prime}}{n^{\prime}}=-(\mu_i-\dfrac{2m}{p+n}-2)\leq 0$$
and $$p+n-\mu_i-\dfrac{2m^{\prime}}{n^{\prime}}=\dfrac{2m}{p+n}-\mu_i+2-p< 0.$$
Similarly, $2p+n-\lambda_i-\dfrac{2m^{\prime}}{n^{\prime}}\leq 0$ and
$p+n-\lambda_i-\dfrac{2m^{\prime}}{n^{\prime}}< 0.$\\
\indent It is now easy to see that $LE(\tilde{G_1})$=$LE(\tilde{G_2})$.
\end{proof}
\indent The following is a consequence of Theorem 4.6.
\begin{corollary}
Let $G_1(s,m)$ and $G_2(s,m)$ be two connected proper subgraphs of the complete graph $K_n$ having $L$-spectra respectively$0=\mu_s<\mu_{s-1}\leq\cdots\leq\mu_1$ and $0=\lambda_s<\lambda_{s-1}\leq\cdots\leq\lambda_1,$ with algebraic connectivity greater than two. For a positive integer $p$ such that  $$\frac{2m}{p+n}<\min(\mu_{n-1}-2, \lambda_{n-1}-2),$$ we have $LE(\tilde{G_1})$=$LE(\tilde{G_2})$, where $(G_i=(K_n-E(G_1))\vee K_p)\times K_p$ for $i=1, 2$.
\end{corollary}
\begin{theorem}
Let $G_1$ and $G_2$ be two connected graphs having $L$-spectra respectively $0=\mu_n<\mu_{n-1}\leq\cdots\leq\mu_1$ and $0=\lambda_n<\lambda_{n-1}\leq\cdots\leq\lambda_1$,.
For a positive integer $p\geq n+4$ such that
$$\frac{2m}{p+n}<\min(\mu_{n-1}, \lambda_{n-1}),$$ we have $LE(\tilde{G_1})$=$LE(\tilde{G_2})$, where $\tilde{G_i}=(G_i\vee\bar{K_p})\times K_p$, $i=1, 2$.
\end{theorem}
\begin{proof}
The $L$-spectra of the graphs $G_1\vee\bar{K_p}$ and $G_2\vee\bar{K_p}$ are respectively as
$\gamma_1, \gamma_2, \cdots, \gamma_{p+n}$ and $\theta_1, \theta_2, \cdots, \theta_{p+n},$ where
$$\gamma_i=\left\{\begin{array}{lr}p+\mu_1, &\mbox{if $1\leq i \leq {n-1}$};\\
p+n, &\mbox{if $i=n$};\\
n, &\mbox{if $n+1\leq i \leq{p+n-1}$};\\
0, &\mbox{if $i=p+n.$}
\end{array} \right.$$
and
$$\theta_i=\left\{\begin{array}{lr}p+\lambda_i, &\mbox{if $1\leq i \leq{n-1}$};\\
p+n, &\mbox{if $i=n$};\\
n, &\mbox{if ${n+1}\leq i \leq{p+n-1}$};\\
0, &\mbox{if $i=p+n.$}
\end{array} \right.$$
Also, $L$-spectra of $K_p$ is
$p=\tau_1=\tau_2=\cdots=\tau_{p-1}, 0=\tau_p.$
So, by Lemma 1.1, the $L$-spectra of $\tilde{G_1}$ and $\tilde{G_2}$ is
$\gamma_i+\tau_j$ and $\theta_i+\tau_j$, where $1\leq i \leq{p+n}, 1\leq j \leq p,$
with average vertex degree $\dfrac{2m^{\prime}}{n^{\prime}}=p-1+\dfrac{2m}{n+p}+\dfrac{2pn}{p+n}.$
For $i=1, 2, \cdots, n-1$, we have
$$2p+\mu_i-\dfrac{2m^{\prime}}{n^{\prime}}=\dfrac{p(p-n)}{p+n}+\mu_i-\dfrac{2m}{p+n}+1\geq 0$$ and
$$p+\mu_i-\dfrac{2m^{\prime}}{n^{\prime}}=\dfrac{(p+n)\mu_i-2pn}{p+n}+\dfrac{(p+n)-2m}{p+n}\leq 0.$$
Similarly, $2p+\lambda_i-\dfrac{2m^{\prime}}{n^{\prime}}\geq 0$ and
$p+\lambda_i-\dfrac{2m^{\prime}}{n^{\prime}}\leq 0.$ \\
Therefore,
\begin{align*}
&LE(\tilde{G_1})=\sum\limits_{i=1}^{p+n}\sum\limits_{j=1}^{p}|\gamma_i+\tau_j-\dfrac{2m^{\prime}}{n^{\prime}}|\\&
=(p-1)\sum\limits_{i=1}^{n-1}(2p+\mu_i-\dfrac{2m^{\prime}}{n^{\prime}})+(p-1)\sum\limits_{i=n+1}^{p+n-1}|p+n-\dfrac{2m^{\prime}}{n^{\prime}}|\\&
+(p-1)|2p+n-\dfrac{2m^{\prime}}{n^{\prime}}|+\sum\limits_{i=1}^{n+p}|\gamma_i-\dfrac{2m^{\prime}}{n^{\prime}}|+(p-1)|p-\dfrac{2m^{\prime}}{n^{\prime}}|\\&
=(p-1)(n-1)(2p-\dfrac{2m^{\prime}}{n^{\prime}})+2m(p-2)+(p-1)^2|p+n-\dfrac{2m^{\prime}}{n^{\prime}}|\\&+(p-1)|2p+n-\dfrac{2m^{\prime}}{n^{\prime}}|
+(p-1)|p-\dfrac{2m^{\prime}}{n^{\prime}}|+(n-1)(\dfrac{2m^{\prime}}{n^{\prime}}-p)\\&+(p-1)|n-\dfrac{2m^{\prime}}{n^{\prime}}|+|p+n-\dfrac{2m^{\prime}}{n^{\prime}}|+\dfrac{2m^{\prime}}{n^{\prime}}.
\end{align*}
Also,
\begin{align*}
&LE(\tilde{G_2})=\sum\limits_{i=1}^{p+n}\sum\limits_{j=1}^{p}|\theta_i+\tau_j-\dfrac{2m^{\prime}}{n^{\prime}}|\\&
=(p-1)(n-1)(2p-\dfrac{2m^{\prime}}{n^{\prime}})+2m(p-2)+(p-1)^2|p+n-\dfrac{2m^{\prime}}{n^{\prime}}|\\&+(p-1)|2p+n-\dfrac{2m^{\prime}}{n^{\prime}}|
+(p-1)|p-\dfrac{2m^{\prime}}{n^{\prime}}|+(n-1)(\dfrac{2m^{\prime}}{n^{\prime}}-p)\\&+(p-1)|n-\dfrac{2m^{\prime}}{n^{\prime}}|+|p+n-\dfrac{2m^{\prime}}{n^{\prime}}|+\dfrac{2m^{\prime}}{n^{\prime}}.
\end{align*}
\indent Clearly, $LE(\tilde{G_1})=LE(\tilde{G_2})$
\end{proof}
\indent From Theorem 2.1, we have the following observation.
\begin{corollary}
Let $G_1$ and $G_2$ be two connected graphs having L-spectra respectively $0=\mu_n<\mu_{n-1}\leq\cdots\leq\mu_1$ and $0=\lambda_n<\lambda_{n-1}\leq\cdots\leq\lambda_1$,.
For a positive integer $p\geq n$ such that  $$\frac{2m}{p+n}<\min(\mu_{n-1}, \lambda_{n-1}),$$ we have $LE((G_1\vee
K_p)\cup\bar{K_p})=LE((G_2\vee K_p)\cup\bar{K_p})$.
\end{corollary}
\begin{proof}
The $L$-spectra of the graphs $(G_1\vee K_p)\cup\bar{K_p}$ and $(G_2\vee K_p)\cup\bar{K_p}$ are respectively
$p+\mu_1, p+\mu_2, \cdots, p+\mu_{n-1}, (p+n) (p-times),0 (p-times)$ and
$p+\lambda_1, p+\lambda_2, \cdots, p+\lambda_{n-1}, (p+n) (p-times), 0 (p-times),$
with average vertex degree
$\dfrac{2m^{\prime}}{n^{\prime}}=\dfrac{2pn}{2p+n}+\dfrac{2m}{2p+n}+\dfrac{p(p-1)}{2p+n}.$
So for $i=1, 2, \cdots, n-1$, we have
$$p+n-\dfrac{2m^{\prime}}{n^{\prime}}=\dfrac{p(p-n+1)}{2p+n}+\mu_i-\dfrac{2m}{n+2p}\geq 0$$ and
$$p+\lambda_i-\dfrac{2m^{\prime}}{n^{\prime}}\geq 0.$$
Therefore proceeding as in Theorem 2.1, it can be seen that $LE((G_1\vee K_p)\cup\bar{K_p})=LE((G_2\vee K_p)\cup\bar{K_p})$.
\end{proof}
\indent For any two graphs $G_1$ and $G_2$ having same number of vertices and edges, it is not always true that  $LE(G_1\times K_p)=LE(G_2\times K_p)$. As an example consider the graphs $G_1$ and $G_2$ (Figure 1). By direct calculation it can be seen that for $p=4$
, $LE(G_1\times K_p)=41.70818\neq 42.05078=LE(G_2\times K_p)$; for $p=7$, $LE(G_1\times K_p)=83.4164\neq 84.1016=LE(G_2\times K_p)$; and for $p=6$,
$LE(G_1\times K_p)=69.5139\neq 70.0849=LE(G_2\times K_p).$ However, for a certain class of graphs, whose algebraic connectivity when added to one becomes greater or equal to average vertex degree the above equality holds for $p> n$, which can be seen in the following result.\\

\unitlength 1mm 
\linethickness{0.4pt}
\ifx\plotpoint\undefined\newsavebox{\plotpoint}\fi 
\begin{picture}(89.5,37.75)(0,0)
\put(11,32.5){\line(1,0){19}}
\put(11,32.5){\line(0,-1){8}}
\put(11,24.5){\line(1,0){17.5}}
\put(11.25,24.25){\line(0,-1){6.75}}
\put(28.25,24.5){\line(0,-1){7}}
\put(48.5,32.5){\line(1,0){25.5}}
\put(74,32.5){\line(0,-1){9}}
\put(74,23.75){\line(1,0){10.75}}
\put(59.25,32.25){\line(0,-1){7.25}}
\put(11,32.5){\circle*{1.12}}
\put(11.25,24.5){\circle*{1.12}}
\put(11.25,17.75){\circle*{1.58}}
\put(28.25,24){\circle*{1}}
\put(28.5,17.75){\circle*{1}}
\put(30,32){\circle*{1}}
\put(48.5,32.25){\circle*{1.5}}
\put(59,32.25){\circle*{1.12}}
\put(59.25,25){\circle*{1.12}}
\put(74,23.75){\circle*{1.5}}
\put(84.5,23.5){\circle*{1.5}}
\put(73.75,32.25){\circle*{1.58}}
\put(42.25,5.25){\makebox(0,0)[cc]{Figure 1}}
\put(29,36.25){\makebox(0,0)[cc]{$v_2$}}
\put(8.5,24.25){\makebox(0,0)[cc]{$v_3$}}
\put(31.25,25){\makebox(0,0)[cc]{$v_4$}}
\put(10.75,15.5){\makebox(0,0)[cc]{$v_5$}}
\put(29.5,16.75){\makebox(0,0)[cc]{}}
\put(32.25,17.5){\makebox(0,0)[cc]{$v_6$}}
\put(20,11.75){\makebox(0,0)[cc]{$G_1$}}
\put(47.75,37){\makebox(0,0)[cc]{$u_1$}}
\put(58.75,37){\makebox(0,0)[cc]{$u_2$}}
\put(74,36.5){\makebox(0,0)[cc]{$u_3$}}
\put(59.25,20.5){\makebox(0,0)[cc]{$u_4$}}
\put(74,19){\makebox(0,0)[cc]{$u_5$}}
\put(89.5,23.25){\makebox(0,0)[cc]{$u_6$}}
\put(69.25,11){\makebox(0,0)[cc]{$G_2$}}
\put(10.25,37.75){\makebox(0,0)[cc]{$v_1$}}
\end{picture}
\begin{theorem}
Let $G_1$ and $G_2$ be two connected graphs having $L$-spectra respectively $0=\mu_n<\mu_{n-1}\leq\cdots\leq\mu_1$ and $0=\lambda_n<\lambda_{n-1}\leq\cdots\leq\lambda_1$, with $\min(\mu_{n-1}, \lambda_{n-1})\geq \dfrac{2m}{n}+1$, then for $p> n$, $LE(G_1\times K_p)=LE(G_2\times K_p)$.
\end{theorem}
\begin{proof}
The $L$-spectra of the graphs $G_1\times K_p$ and $G_2\times K_p$ are respectively
$\mu_i+\tau_j$ and $\lambda_i+\tau_j$, where $1\leq i \leq n, 1\leq j \leq p$ and $\tau_j$ is as in Theorem 4.8.
With average degree $\dfrac{2m^{\prime}}{n^{\prime}}=\dfrac{2m}{n}+p-1.$
Therefore,
\begin{align*}
LE(G_1\times K_p)&
=\sum\limits_{i=1}^{n}\sum\limits_{j=1}^{p}|\mu_i+\tau_j-\dfrac{2m^{\prime}}{n^{\prime}}|\\&
=(p-1)\sum\limits_{i=1}^{n}|\mu_i+p-\dfrac{2m}{n}-p+1|+\sum\limits_{i=1}^{n}|\mu_i-\dfrac{2m}{n}-p+1|\\&
=(p-1)\sum\limits_{i=1}^{n}(\mu_i-\dfrac{2m}{n}+1)+\sum\limits_{i=1}^{n}(\dfrac{2m}{n}+p-1-\mu_i)\\&
=2m(p-2)+n(p-2)(1-\dfrac{2m}{n})+pn.
\end{align*}
Also
\begin{align*}
LE(G_2\times K_p)&
=\sum\limits_{i=1}^{n}\sum\limits_{j=1}^{p}|\lambda_i+\tau_j-\dfrac{2m^{\prime}}{n^{\prime}}|\\&
=2m(p-2)+n(p-2)(1-\dfrac{2m}{n})+pn\\&
=LE(G_1\times K_p).
\end{align*}
\end{proof}
\indent We have seen that given a pair of graphs it is always possible to obtain a sequence of graph pairs having same $L-energy$, and thus we have the following observation.
\begin{theorem}
Let $G_1(n,m)$ and $G_1^{\prime}(n,m)$ be two $L$-equienergetic graphs having $L$-spectra respectively $0=\mu_n\leq\mu_{n-1}\leq\cdots\mu_1$ and $0=\mu_n^{\prime}\leq\mu_{n_1}^{\prime}\leq\cdots\mu_1^{\prime},$ and let $G_2(n,m)$ and $G_2^{\prime}(n,m)$ be another pair of $L$-equienergetic graphs having $L$-spectra respectively
$0=\lambda_n\leq\lambda_{n-1}\leq\cdots\lambda_1$ and $0=\lambda_n^{\prime}\leq\lambda_{n_1}^{\prime}\leq\cdots\lambda_1^{\prime},$ then $LE(G_1\vee G_2)=LE(G_1^{\prime}\vee G_2^{\prime})$.
\end{theorem}
\begin{proof}
If $H=G_1\vee G_2$ and $K=G_{1}^{\prime}\vee G_{2}^{\prime}$, then the $L$-spectra of $H$ and $K$ are respectively
$2n, n+\mu_1, \cdots, n+\mu_{n-1}, n+\lambda_1, n+\lambda_2, \cdots, n+\lambda_{n-1}, 0$ and
$2n, n+\mu_1^{\prime}, \cdots, n+\mu_{n-1}^{\prime}, n+\lambda_1^{\prime}, n+\lambda_2^{\prime}, \cdots, n+\lambda_{n-1}^{\prime}, 0,$
with average vertex degree $\dfrac{2m^{\prime}}{n^{\prime}}=\dfrac{2m+n^2}{n}.$
Since ($G_1$, $G_{1}^{\prime}$) and ($G_2$, $G_{2}^{\prime}$) are $L$-equienergetic graph pairs, therefore $LE(G_1)=LE(G_{1}^{\prime})$ and  $LE(G_2)=LE(G_{2}^{\prime}).$ Now,
\begin{align*}
LE(H)&
=|2n-\dfrac{2m+n^2}{n}|+\sum\limits_{i=1}^{n-1}|n+\mu_i-\dfrac{2m+n^2}{n}|\\&+\sum\limits_{i=1}^{n-1}|n+\lambda_i-\dfrac{2m+n^2}{n}|+|0-\dfrac{2m+n^2}{n}|\\&
=2n+LE(G_1)+LE(G_2)-\dfrac{4m}{n}
\end{align*}
and
\begin{align*}
LE(K)&=|2n-\dfrac{2m+n^2}{n}|+\sum\limits_{i=1}^{n-1}|n+\mu_i^{\prime}-\dfrac{2m+n^2}{n}|\\&+\sum\limits_{i=1}^{n-1}|n+\lambda_i^{\prime}-\dfrac{2m+n^2}{n}|+|0-\dfrac{2m+n^2}{n}|\\&
=2n+LE(G_{1}^{\prime})+LE(G_{2}^{\prime})-\dfrac{4m}{n}.
\end{align*}
Clearly the result follows.
\end{proof}
\indent The next result is a generalization of Theorem 4.11.
\begin{theorem}
Let $(G_1, G_{1}^{\prime}),(G_2, G_{2}^{\prime}),\cdots, (G_k, G_{k}^{\prime})$ be $k$ $L$-equienergetic graph pairs with same number of vertices $n$ and edges $m$, then
$LE(G_1\vee G_2\vee\cdots\vee G_k)=LE(G_{1}^{\prime}\vee G_{2}^{\prime}\vee\cdots\vee G_{k}^{\prime})$.
\end{theorem}
\begin{proof}
Let $0=\mu_{ni}\leq\mu_{(n-1)i}\leq\cdots\mu_{1i}$ and $0=\mu_{ni}^{\prime}\leq\mu_{(n-1)i}^{\prime}\leq\cdots\mu_{1i}^{\prime}$ be the $L$-spectra of the graphs $G_i$ and $G_{i}^{\prime}$, respectively. Then applying Lemma 1.3 repeatedly we find that the $L$-spectra of the graphs $(G_1\vee G_2\vee \cdots \vee G_k)$ and $(G_{1}^{\prime}\vee
G_{2}^{\prime} \vee \cdots \vee G_{k}^{\prime})$ are
$kn ((k-1)-times), n(k-1)+\mu_{1i}, \cdots, n(k-1)+\mu_{(n-1)i}, 0$ and
$kn ((k-1)-times), n(k-1)+\mu_{1i}^{\prime}, \cdots, n(k-1)+\mu_{(n-1)i}^{\prime}, 0,$ respectively,
with average vertex degree $\dfrac{2m^{\prime}}{n^{\prime}}=\dfrac{2m+n^2(k-1)}{n}.$\\
Therefore,
\begin{align*}
&LE((G_1\vee G_2\vee \cdots \vee G_k))=(k-1)|kn-\dfrac{(k-1)2m+n^2}{n}|
\\&+\sum\limits_{i=1}^{k}\sum\limits_{j=1}^{n}|n(k-1)+\mu_{ji}-\dfrac{2m+n^2(k-1)}{n}|+|0-\dfrac{2m+n^2(k-1)}{n}|\\&
=(k-1)(kn-\dfrac{2m+n^2(k-1)}{n})+\sum\limits_{i=1}^{k}(LE(G_i)-\dfrac{2m}{n})+\dfrac{2m+n^2(k-1)}{n}\\&
=\sum\limits_{i=1}^{k}LE(G_i)+2n(k-1)-(2k-2)\dfrac{2m}{n}.
\end{align*}
Also,
\begin{align*}
&LE((G_{1}^{\prime}\vee
G_{2}^{\prime} \vee \cdots \vee G_{k}^{\prime}))=(k-1)|kn-\dfrac{2m+n^2(k-1)}{n}|
\\&+\sum\limits_{i=1}^{k}\sum\limits_{j=1}^{n}|n(k-1)+\mu_{ji}^{\prime}-\dfrac{2m+n^2(k-1)}{n}|+|0-\dfrac{2m+n^2(k-1)}{n}|\\&
=\sum\limits_{i=1}^{k}LE(G_{i}^{\prime})+2n(k-1)-(2k-2)\dfrac{2m}{n}.
\end{align*}
Clearly, the result follows.
\end{proof}
\indent Using the Kronecker product of graphs, we now give another construction of sequences of $L$-equienergetic ($Q$-equienergetic) graphs from a given pair of connected graphs.\\
\begin{theorem}
Let $G_1$ and $G_2$ be two connected graphs having $L$-spectra respectively $0=\mu_n<\mu_{n-1}\leq\cdots\leq\mu_1$ and $0=\lambda_n<\lambda_{n-1}\leq\cdots\leq\lambda_1$.
For a positive integer $p\geq n$ such that\\
$$\frac{2m}{p+n}<\min(\mu_{n-1}, \lambda_{n-1}),$$ we have $LE(\tilde{G_1})$=$LE(\tilde{G_2})$, where $\tilde{G_i}=(G_i\vee \bar{K_p})\otimes K_p$, $i=1, 2$.
\end{theorem}
\begin{proof}
The $L$-spectra of the graphs $G_1\vee \bar{K_p}$ and $G_2\vee \bar{K_p}$ are respectively
$\gamma_1, \gamma_2, \cdots, \gamma_{p+n}$ and $\theta_1, \theta_2, \cdots, \theta_{p+n}$, where
$$\gamma_i=\left\{\begin{array}{lr}p+\mu_1, &\mbox{if $1\leq i \leq {n-1}$};\\
p+n, &\mbox{if $i=n$};\\
n, &\mbox{if ${n+1}\leq i \leq{p+n-1}$};\\
0, &\mbox{if $i=p+n$}.
\end{array} \right.$$
and
$$\theta_i=\left\{\begin{array}{lr}p+\lambda_i, &\mbox{if $1\leq i \leq{n-1}$};\\
p+n, &\mbox{if $i=n$};\\
n, &\mbox{if $n+1\leq i \leq{p+n-1}$};\\
0, &\mbox{if $i=p+n$}.
\end{array} \right.$$
Also, $L$-spectra of $K_p$ is $p=\tau_1=\tau_2=\cdots=\tau_{p-1}, 0=\tau_p.$
So the $L$-spectra of $\tilde{G_1}$ and $\tilde{G_2}$ is
$\gamma_i\tau_j$ and $\theta_i\tau_j$, where $1\leq i \leq{p+n}, 1\leq j \leq p,$
with average vertex degree $\dfrac{2m^{\prime}}{n^{\prime}}=\dfrac{2m(p-1)}{n+p}+\dfrac{2pn(p-1)}{p+n}.$
For $i=1, 2, \cdots, n-1$, we have
$$p(p+\mu_i)-\dfrac{2m^{\prime}}{n^{\prime}}=p\mu_i-\dfrac{2m(p-1)}{p+n}+\dfrac{p(p^2-pn+2n)}{p+n}\geq 0.$$
Similarly $p(p+\lambda_i)-\dfrac{2m^{\prime}}{n^{\prime}}\geq 0.$
\indent It is now easy to see that $LE(\tilde{G_1})$=$LE(\tilde{G_2})$.
\end{proof}
\indent Theorems 4.14 and 4.15 are concerned with the construction of disconnected graphs having same $L$-energy ($Q$-energy) from a given pair of connected graphs.
\begin{theorem}
Let $G_1$ and $G_2$ be two connected graphs having $L$-spectra respectively $0=\mu_n<\mu_{n-1}\leq\cdots\leq\mu_1$ and $0=\lambda_n<\lambda_{n-1}\leq\cdots\leq\lambda_1$.
For a positive integer $p$ such that
$$\frac{2m}{p+n}<\min(\mu_{n-1}, \lambda_{n-1}),$$ we have $LE(\tilde{G_1})$=$LE(\tilde{G_2})$, where $\tilde{G_i}=(G_i\cup \bar{K_p})\otimes K_p$, $i=1, 2$.
\end{theorem}
\begin{proof}
The $L$-spectra of the graphs $\tilde{G_1}$ and $\tilde{G_2}$ are respectively given as
$p\mu_1, p\mu_2, \cdots, p\mu_{n-1} (each (p-1)-times), 0 ((p^2+p+n-1)-times)$ and
$p\lambda_1, p\lambda_2, \cdots, p\lambda_{n-1} (each (p-1)-times), 0 ((p^2+p+n-1)-times),$
with average vertex degree $\dfrac{2m^{\prime}}{n^{\prime}}=\dfrac{2m(p-1)}{p+n}.$
So, for $i=1, 2, \cdots, n-1$, we have
$$p\mu_i-\dfrac{2m^{\prime}}{n^{\prime}}\geq 0$$ and
$$p\lambda_i-\dfrac{2m^{\prime}}{n^{\prime}}\geq 0.$$
Now, it is easy to see that $LE(\tilde{G_1})$=$LE(\tilde{G_2})$.
\end{proof}
\begin{theorem}
Let $G_1$ and $G_2$ be two connected graphs having $Q$-spectra respectively $0<\mu^+_n<\mu^+_{n-1}\leq\cdots\leq\mu^+_1$ and $0<\lambda^+_n<\lambda^+_{n-1}\leq\cdots\leq\lambda^+_1$.
For a positive integer $p> n$ such that  $$\frac{2m}{p+n}<\min(\mu^+_{n}, \lambda^+_{n}),$$ we have $LE^+((G_1\cup\bar{K_p})\otimes K_p)=LE^+((G_2\cup\bar{K_p})\otimes K_p)$.
\end{theorem}
\begin{proof}
The proof follows from Theorem 4.14 and the fact that the $Q$-spectra of $K_p$ is $2p-2, p-2 ((p-1)-times)$.
\end{proof}

\end{document}